\documentclass[english,10pt]{smfart}
\usepackage{url,verbatim}
\usepackage{mathrsfs}
\usepackage[all,2cell,emtex]{xy}
\usepackage[colorlinks,linkcolor=blue,citecolor=blue,urlcolor=red]{hyperref}

\newtheorem{thm}{Theorem}[section]
\newtheorem{prop}[thm]{Proposition}

\newtheorem{cor}[thm]{Corollary}

\theoremstyle{remark} 

\newtheorem{paragr}[thm]{}

\theoremstyle{definition} 
\newtheorem{definition}[thm]{Definition}

\newcommand{\by}[1]{\overset{#1}{\longrightarrow}}

\newcommand{\To}{\longrightarrow}

\newcommand{\Hom}{\operatorname{\mathrm{Hom}}}
\newcommand{\sHom}{\mathbf{Hom}}

\newcommand{\ho}{\operatorname{\mathrm{Ho}}}

\renewcommand{\AA}{\mathbf{A}}
\newcommand{\PP}{\mathbf{P}}

\newcommand{\spec}[1]{\mathrm{Spec}(#1)}
\newcommand{\Sm}[1]{\mathit{Sm}_{/#1}}

\newcommand{\AH}{\mathscr{H}_{\AA^1}}
\newcommand{\birat}{b}
\newcommand{\birH}{\mathscr{H}_{\birat}}
\newcommand{\spaces}[1]{\mathit{sPr}({\Sm #1})}
\newcommand{\Lbir}{\operatorname{L}_\birat}
\newcommand{\LA}{\operatorname{L}_{\AA^1}}
\newcommand{\god}{\mathscr{G}^\bullet}
\newcommand{\sing}{\mathit{Sing}_{\AA^1}}
\newcommand{\exA}{\mathit{Ex}_{\AA^1}}

\def\TO#1{\mathrel{\hbox to #1mm{\rightarrowfill}}}
\def\OT#1{\mathrel{\hbox to #1mm{\leftarrowfill}}}

\renewcommand{\theequation}{\arabic{equation}}
\numberwithin{equation}{thm}

\title[Homotopy theory of schemes and $R$-equivalence]{Homotopy theory of schemes\\ and $R$-equivalence}
\author[D.-C. Cisinski]{Denis-Charles Cisinski}
\address{Fakult\"at f\"ur Mathematik\\
Universit\"at Regensburg\\
93040 Regensburg\\
Deutschland}
\email{denis-charles.cisinski@mathematik.uni-regensburg.de}
\urladdr{http://www.mathematik.uni-regensburg.de/cisinski/index.html}
\author[B. Kahn]{Bruno Kahn}
\address{CNRS, Sorbonne Université and Université Paris Cité, IMJ-PRG\\ Case 247\\4 place
Jussieu\\75252 Paris Cedex 05\\France}
\email{bruno.kahn@imj-prg.fr}
\urladdr{https://webusers.imj-prg.fr/~bruno.kahn/}
\begin{document}
\begin{abstract}
We prove that, for any smooth and projective scheme $X$ over a field $k$ of char.~$0$,
the set of maps from $\spec k$ to $X$ in the $\AA^1$-homotopy category of schemes $\AH(k)$
is in bijection with the quotient of $X(k)$ by $R$-equivalence, and is a birational
invariant of $X$. This is achieved by establishing a precise relation between the
localization of the category of smooth $k$-schemes by birational maps and the
category $\AH(k)$, and by applying results of the second named author and
R.~Sujatha on birational invariants. This gives a new proof of results obtained
by A.~Asok and F.~Morel \cite{AM}.
\end{abstract}

\maketitle
\tableofcontents

This manuscript was written in 2010, and updated in 2017. We make it public due to the interest of some mathematicians.

We fix once and for all a field $k$. By a $k$-variety, we shall mean an
integral separated $k$-scheme of finite type. Given a $k$-variety $X$,
$k(X)$ will stand for its field of functions.

\section{The birational homotopy category}

\begin{definition}[Manin]
Let $X$ be a $k$-scheme of finite type.

Two rational points $x_0$ and $x_1$ of $X$ are \emph{directly $R$-equivalent}
if there exists a rational map $f:\PP^1_k\dashrightarrow X$, defined at $0$ and $1$, such that
$f(0)=x_0$ and $f(x_1)=x_1$.

The \emph{$R$-equivalence} on $X(k)$ is the equivalence relation
generated on $X(k)$ by the direct $R$-equivalence.
\end{definition}

\begin{paragr}
Let $\Sm k$ be the category of separated and smooth $k$-schemes of finite type.
We denote by $S_b$ be the class of birational morphisms in $\Sm k$.
We denote by $S^{-1}_b\Sm k$ the localization of $\Sm k$ by $S_b$.
\end{paragr}

\begin{prop}[Colliot-Th\'el\`ene]\label{prop:birA1hmt0}
The projection $\AA^1_{k}\To \spec k$ is an isomorphism in $S^{-1}_b\Sm k$.
\end{prop}

\begin{proof}
As $\AA^1_k\simeq\PP^1_k$ in $S^{-1}_b\Sm k$, it follows from
(the proof of) the main assertion of \cite[Appendix A]{birgeom}.

Let us paraphrase Colliot-Th\'el\`ene's proof, in order to show its analogy with the proof of
\cite[Prop. 2.2.3 a)]{birat1}. Let $\phi$ be the birational isomorphism between $\PP^2$ and
$\PP^1\times \PP^1$, normalized via the isomorphism $\AA^2\by{\sim} \AA^1\times \AA^1$. Let $W$
be the (closure of the graph of) $\phi$: this is a del Pezzo surface, the blow-up of $\PP^2$ at
$\{0,\infty\}\subset \PP^1\equiv \PP^2-\AA^2$ (it is the same $W$ as in Colliot-Th\'el\`ene's
proof). Projected onto $\PP^1\times \PP^1$, it is its blow-up at $(\infty,\infty)$. Keeping the
notation of \cite[Appendix A]{birgeom}, we have a commutative diagram
\[\xymatrix{
& E_1\ar[dl]_\pi\ar[d]\ar[dr]^{\sim}\\
M_1\ar[d]& W\ar[dl]_a\ar[dr]^b&L_1\ar[d]\ar[dr]^\sim\\
\PP^2 && \PP^1\times \PP^1\ar[r]^{p_1}& \PP^1.
}\]

Here $M_1$ is one of the points blown-up in $\PP^2$, $E_1$ is its inverse image in $W$ (so that
$E_1\simeq \PP^1$) and $L_1$ is the (isomorphic) projection of $E_1$ in $\PP^1\times \PP^1$.
For any fonctur $F$ which transforms $a$ and $b$ into isomorphisms, the isomorphism
$F(E_1)\by{\sim} F(\PP^1)$ factors through $F(\pi)$; finally $\pi$ has a section, given by any
rational point of $E_1$, hence $F(\pi)$ must be an isomorphism.
\end{proof}

\begin{paragr}
Note that $S_b$ is closed under finite products in $\Sm k$,
so that the  localized category $S^{-1}_b\Sm k$ admits finite products, and
the localization functor $\Sm k\To S^{-1}_b\Sm k$ commutes with them.
The preceding proposition thus implies the following homotopy invariance property.
\end{paragr}

\begin{prop}\label{prop:birA1hmt}
For any separated smooth $k$-scheme of finite type $X$, the
projection $\AA^1_X\To X$ is an isomorphism in $S^{-1}_b\Sm k$.
\end{prop}

\begin{paragr}
For $Y$ a smooth and proper $k$-variety, the direct $R$-equivalence
relation on $Y(k)$ coincides with the $\AA^1$-homotopy relation:
two rational points $y_0$ and $y_1$ are directly $R$-equivalent if and only if
there exists a map $f:\AA^1_k\To Y$ such that $f(0)=y_0$ and $f(1)=y_1$. (This follows from the valuative criterion of properness.) 
It thus follows from Proposition \ref{prop:birA1hmt} that, for
any smooth $k$-varieties $X$ and $Y$, with $Y$ proper, the canonical map
\begin{equation}\label{eq0}
Y(k(X))\To\Hom_{S_b^{-1}\Sm k}(X,Y)
\end{equation}
factors through $R$-equivalence, and induces a map
\begin{equation}\label{eq1}
Y(k(X))/R\To\Hom_{S_b^{-1}\Sm k}(X,Y) \, .
\end{equation}
We then have the following result; see \cite[Theorem 6.6.3]{birgeom}.
\end{paragr}

\begin{thm}\label{thm:Requiv}
Consider two
smooth $k$-varieties $X$ and $Y$, with $Y$ proper.
Then \eqref{eq1} is bijective.
\end{thm}

\begin{paragr}
Let $\spaces k$ be the category of simplicial presheaves on the category $\Sm k$,
endowed with its projective model category structure: the weak equivalences
(resp. the fibrations) are the termwise simplicial weak equivalences (resp. the termwise
Kan fibrations). Note that, for this model category structure, all representable presheaves
are cofibrant.

Considering $S_b$ as a set of maps in $\spaces k$ via the Yoneda embedding,
we shall consider the left Bousfield localization $\Lbir\spaces k$
of $\spaces k$ by $S_b$. The cofibrations of $\Lbir\spaces k$
are thus the same as the cofibrations in $\spaces k$, while the fibrant
objects are the presheaves of Kan complexes $F$ on $\Sm k$, such that, for any
map $X\To Y$ in $S_b$, the induced map $F(Y)\To F(X)$ is a simplicial homotopy equivalence.
Remark that $\Lbir\spaces k$ can also be described as the left Bousfield
localization of $\spaces k$ by the set of dense open immersions, which means
that the fibrant objects of $\Lbir\spaces k$ may as well be described as the
presheaves of Kan complexes on $\Sm k$ which send dense open immersions to
simplicial homotopy equivalences. The weak equivalences of $\Lbir\spaces k$
will be called \emph{birational weak equivalences}. We write
$$\birH(k)=\ho(\Lbir\spaces k)$$
for the homotopy category of $\Lbir\spaces k$.
\end{paragr}

\begin{prop}\label{birhomtopyyoneda}
The Yoneda embedding $\Sm k\To \spaces k$ induces
a fully faithful functor
$$S^{-1}_b\Sm k\To \birH(k)\, .$$
In particular, 
given two smooth $k$-varieties $X$ and $Y$, with $Y$
proper, we have a canonical bijection
$$\Hom_{\birH(k)}(X,Y)\simeq Y(k(X))/R\, .$$
\end{prop}

\begin{proof}
The first assertion is a straightforward consequence of
the general results of Dwyer and Kan on simplicial
presheaves on localizations of simplicial categories: by virtue of
\cite[Theorem 2.2]{DK}, $\Lbir\spaces k$ is (Quillen equivalent to)
the model category of simplicial presheaves on the simplicial
localization of $\Sm k$ by $S_b$. The second assertion follows immediately
from the first and from Theorem \ref{thm:Requiv}.
\end{proof}

\section{Birational localization and $\AA^1$-homotopy theory}

\begin{paragr}
We recall the description of Morel and Voevodsky's
$\AA^1$-homotopy category $\AH(k)$ \cite{MV} as the
homotopy category of a left Bousfield localization of the projective model category
structure on $\spaces k$.

Recall that a \emph{distinguished square} is a pullback square
in $\Sm k$ of the form
\begin{equation}\label{eq:distsq}
\begin{split}
\xymatrix{
U\times_{X} V\ar[r]\ar[d]&V\ar[d]^f\\
U\ar[r]^j&X
}\end{split}\end{equation}
in which $j$ is an open immersion, while $f$ is \'etale and
induces an isomorphism $f^{-1}((X-U)_{\mathit{red}})\simeq (X-U)_{\mathit{red}}$.
For such a square, we shall denote by $U\amalg^p_{U\times_X V} V$
the corresponding \emph{pushout of presheaves} over $\Sm k$.

Let $\LA\spaces k$ be the left Bousfield localization of $\spaces k$
by the projections $\AA^1_X\To X$ for any separated and smooth $k$-scheme $X$, and
by the maps of shape
$$U\amalg^p_{U\times_X V} V\To X$$
for any distinguished square \eqref{eq:distsq} in $\Sm k$.
The fibrant objects of $\LA\spaces k$ are thus the
presheaves of Kan complexes $F$ on $\Sm k$ which are
\emph{$\AA^1$-homotopy invariant}, i.e. such that, for any separated
smooth $k$-scheme $X$, the map
$$F(X)\To F(\AA^1_X)$$
is a simplicial homotopy equivalence, and which
have the \emph{Brown-Gersten property}, i.e. such that, for any
distinguished square \eqref{eq:distsq}, the induced diagram
\begin{equation}\label{eq:BG}
\begin{split}
\xymatrix{
F(X)\ar[r]\ar[d]&F(V)\ar[d]\\
F(U)\ar[r]&F(U\times_{X} V)
}\end{split}\end{equation}
is a homotopy pullback square in the homotopy category of
Kan complexes. The weak equivalences of the
model category $\LA\spaces k$ will be called
\emph{weak $\AA^1$-equivalences}.

The corresponding homotopy category
$\ho(\LA\spaces k)$ is canonically equivalent to the
$\AA^1$-homotopy category $\AH(k)$ of Morel and Voevodsky;
see \cite{blander}.
\end{paragr}

\begin{prop}\label{prop:birlocA1}
Any weak $\AA^1$-equivalence is a weak birational equivalence.
\end{prop}

\begin{proof}
Let $F$ be a presheaf of Kan complexes on $\Sm k$ which sends
dense open immersions to simplicial homotopy equivalences.
It follows immediately from Proposition \ref{prop:birA1hmt} that $F$
is $\AA^1$-homotopy invariant.
Moreover, $F$ obviously has the Brown-Gersten property. This is equivalent to 
our assertion.
\end{proof}

As a byproduct, we get the following consequence from the general yoga of left Bousfield
localizations of model categories \cite{Hir}.

\begin{cor}\label{cor:HbirlocHA1}
The category $\birH(k)$ is the localization of $\AH(k)$ by the class
of weak birational equivalences, and the localization functor
$$\gamma_!: \AH(k)\To \birH(k)$$
has a fully faithful right adjoint
$$\gamma^*:\birH(k)\To \AH(k)$$
whose essential image is spanned by the presheaves of
Kan complexes which send
dense open immersions to simplicial homotopy equivalences.
\end{cor}

\section{Rational points up to $\AA^1$-homotopy}

\begin{paragr}
Let $F$ be a presheaf of sets over $\Sm k$.
Given a separated smooth $k$-scheme $X$, we denote by
$[X,F]_{\AA^1}$ the quotient of $F(X)$ by the $\AA^1$-homotopy
relation. We thus have a canonical map
\begin{equation}\label{eq2}
[X,F]_{\AA^1}\To \Hom_{\AH(k)}(X,F)\, .
\end{equation}
In the case where $X=\spec k$ and $F=Y$ for a smooth and proper $k$-variety $Y$
we have
\begin{equation}
Y(k)/R=[\spec k,Y]_{\AA^1}\, .
\end{equation}
\end{paragr}

\begin{prop}[\protect{\cite[Cor. 2.1.5]{AM}}]\label{prop:A1surj}
For any presheaf of sets $F$ over $\Sm k$, \eqref{eq2}
becomes surjective after Nisnevich sheafification.
\end{prop}

\begin{proof}
The assertion is equivalent to the following: for any henselian local ring $R$ of geometric origin, the map \eqref{eq2} evaluated at $X=\spec R$ is surjective. We shall prove this in a more general situation, by considering
a simplicial presheaf $F$ on $\Sm k$. Let
$$\god:\spaces k\To \spaces k$$
be the Godement resolution functor with respect to Nisnevich points;
see \cite[p. 66]{MV}.
We also have an $\AA^1$-resolution functor
$$\sing:\spaces k\To \spaces k$$
which associates to a simplicial presheaf $F$ the simplicial
presheaf $\sing(F)$ defined by
$$\sing(F)_n=\sHom(\Delta^n_k,F_n)\, ,$$
where $\sHom$ denotes the internal Hom in the category of presheaves, while
$$\Delta^n_k=\spec{k[t_0,\ldots,t_n]/(\sum_i t_i -1)}\simeq \AA^n_k$$
is the algebraic simplex.
By virtue of \cite[Lemma 2.3.20]{MV},
we then have an explicit $\AA^1$-resolution functor $\exA$
(i.e. fibrant resolution functor in the model category $\LA\spaces k$)
defined by
$$\exA=\god\circ (\sing)^\omega\circ \god$$
(remark that the Godement resolution does not detect the difference between
a presheaf and its asociated sheaf, so that it does not make any difference
to work with presheaves or sheaves here).

Given a simplicial presheaf $F$, we then have a natural identification
$$\pi_0(\exA(F)(\spec R))\simeq \Hom_{\AH(k)}(\spec R,F)\, .$$
As $R$ is henselian,  the map
$$\pi_0(F(\spec R))\To \pi_0(\god(F)(\spec R))$$
is bijective. On the other hand, 
the set $\pi_0(\sing(F)(\spec R))$ is the quotient of $F_0(\spec R)$ by the
$\AA^1$-homotopy relation. This implies immediately
the proposition.
\end{proof}

\begin{thm}\label{thm:biratA1connected}
For any smooth and proper $k$-variety $Y$, there is a canonical bijection
$$Y(k)/R\simeq \Hom_{\AH(k)}(\spec k, Y)\, .$$
In particular, the functor
$$Y\longmapsto\Hom_{\AH(k)}(\spec k,Y)$$
is a birational invariant of smooth and proper $k$-schemes.
\end{thm}

\begin{proof}
The localization functor $\gamma_!$ of Corollary \ref{cor:HbirlocHA1}
induces a natural map
$$\Hom_{\AH(k)}(\spec k,Y)\To \Hom_{\birH(k)}(\spec k,Y)\, .$$
On the other hand, by virtue of Theorem \ref{thm:Requiv}, the composed map
$$[\spec k,Y]_{\AA^1}=Y(k)/R\To \Hom_{\AH(k)}(\spec k,Y)\To \Hom_{\birH(k)}(\spec k,Y)$$
is bijective. Therefore, by Proposition \ref{prop:A1surj}, we obtain that the map
$$Y(k)/R\To \Hom_{\AH(k)}(\spec k,Y)$$
is bijective.
\end{proof}

\begin{paragr}
Let $F$ be a simplicial presheaf over $\Sm k$. We denote by $\pi^{\AA^1}_0(F)$ the
Nisnevich sheaf associated to the presheaf
$$U\longmapsto\Hom_{\AH(k)}(U,F)\, .$$
\end{paragr}

\begin{cor}\label{cor:biratA1connected}
For any smooth and proper $k$-variety $Y$, and
for any field extension of finite type $L/k$, there is a natural identification
$$Y(L)/R\simeq \pi^{\AA^1}_0(Y)(L)\, .$$
\end{cor}

\section{Relationship with the work of Asok and Morel}

Theorem \ref{thm:biratA1connected} and Corollary \ref{cor:biratA1connected}
are special cases of results of Aravind Asok and Fabien Morel
\cite[Theorem 2.4.3 and Corollary 2.4.9]{AM}, with the difference that
we used Proposition \ref{prop:birA1hmt} to see that, in a precise sense, there is
no difference in working up to birational equivalence or up to
stable birational equivalence (which is stated in a conceptual way
as Corollary \ref{cor:HbirlocHA1}). Let us clarify the relationship between their work and ours:

Asok and Morel consider a composition
\begin{equation}\label{eq3}
\pi_0^{ch}(X)\to \pi_0^{\AA^1}(X)\to \pi_0^{b\AA^1}(X)
\end{equation}
for any proper scheme $X$ over a field, where $\pi_0^{ch}(X)$ is the sheaf associated to the
presheaf of ``na\"\i ve $\AA^1$-equivalence classes" (so that $\pi_0^{ch}(X)(K)= X(K)/\AA^1 =
X(K)/R$ for any function field $K$), and $\pi_0^{b\AA^1}(X)$ is defined \emph{a priori} as a
birational sheaf:
\begin{gather*}
\pi_0^{b\AA^1}(X)(U) = \pi_0^{b\AA^1}(X)(\spec{k(U)}\\
\pi_0^{b\AA^1}(X)(K) = X(K)/R.
\end{gather*}

(The theorem is that $\pi_0^{b\AA^1}(X)(U)$ has the structure of a presheaf in $U$.) Like us
(cf. Proposition \ref{prop:A1surj}), they show that when evaluated at a function field, the
first map in \eqref{eq3} is surjective, hence all maps are bijective since the composition is
the identity.

We consider instead a composition
\begin{equation}\label{eq4}
\pi_0^{ch}(X)\to \pi_0^{\AA^1}(X)\to \pi_0^{b}(X)
\end{equation}
for any smooth scheme $X$, where $\pi_0^b(X)$ is the sheaf associated to the presheaf $U\mapsto
\Hom_{\birH(k)}(U,X)$. Here, evaluated at a function field, the composition is bijective for $X$ proper by Proposition \ref{birhomtopyyoneda}.


\providecommand{\bysame}{\leavevmode\hbox to3em{\hrulefill}\thinspace}
\providecommand{\MR}{\relax\ifhmode\unskip\space\fi MR }
\providecommand{\MRhref}[2]{%
\href{http://www.ams.org/mathscinet-getitem?mr=#1}{#2}
}
\providecommand{\href}[2]{#2}


\begin{thebibliography}{MV99}
\bibitem[AM08]{AM} A.~Asok \& F.~Morel, \emph{Smooth varieties up to $\AA^1$-homotopy
and algebraic $h$-cobordisms}, Adv. in Math. {\bf 227} (2011), 1990--2058.
\bibitem[Bla01]{blander} B.~Blander, \emph{Local projective model structures on simplicial presheaves},
K-theory \textbf{24} (2001), 283--301.
\bibitem[DK87]{DK}
W.~G.~Dwyer \& D.~M.~Kan, \emph{Equivalences between homotopy theories of diagrams},
Algebraic topology and algebraic K-theory, Annals of Math. Studies, vol.~113,
Princeton University Press, 1987, pp.~180--205.
\bibitem[Hir03]{Hir}
P.~S.~Hirschhorn, \emph{Model categories and their localizations},
Math. Surv. and Monographs, vol.~99, Amer; Math. Soc., 2003.
\bibitem[KS08]{birgeom}
B.~Kahn \& R.~Sujatha, \emph{Birational geometry and localisation of categories},
Doc. Math. - Extra Volume Merkurjev (2015), 277--334.
\bibitem[KS09]{birat1} B.~Kahn \& R.~Sujatha, \emph{Birational motives, I: pure birational
motives}, Annals of K-theory 1 (2016), 379--440.
\bibitem[MV99]{MV}
F.~Morel \& V.~Voevodsky, \emph{$\AA^1$-homotopy theory of schemes},
Publ. Math. IHES \textbf{90} (1999), 45--143.
\end{thebibliography}
\end{document}